\documentclass[11pt,letterpaper]{article}
\usepackage[english]{babel}
\usepackage{amsmath}
\usepackage{amsfonts}
\usepackage{amssymb}
\usepackage{verbatim}
\usepackage{graphicx}
\usepackage{color}

\numberwithin{equation}{section}
\newtheorem{theo}{Theorem}[section]
\newtheorem{cor}[theo]{Corollary}
\newtheorem{prop}[theo]{Proposition}
\newtheorem{lemma}[theo]{Lemma}

\newtheorem{remark}[theo]{Remark}
\newtheorem{remarks}[theo]{Remarks}
\newenvironment{proof}[1][Proof]{\textbf{#1.} }{\ \rule{0.5em}{0.5em}}

\begin{document}

\title{Maximums on Trees}

\author{Predrag R. Jelenkovi\'c \\ {\small Department of Electrical Engineering} \\ {\small Columbia University} 
\and
Mariana Olvera-Cravioto \\ {\small Department of Industrial Engineering and Operations Research} \\ {\small Columbia University}
}

\maketitle

\begin{abstract}
We study the minimal/endogenous solution $R$ to the maximum recursion on weighted branching trees given by
$$R\stackrel{\mathcal{D}}{=}\left(\bigvee_{i=1}^NC_iR_i \right)\vee Q,$$
where $(Q,N,C_1,C_2,\dots)$ is a random vector with $N\in \mathbb{N}\cup\{\infty\}$, $P(|Q|>0)>0$ and nonnegative weights $\{C_i\}$, and $\{R_i\}_{i\in\mathbb{N}}$ is a sequence of i.i.d. copies of $R$ independent of $(Q,N,C_1,C_2,\dots)$; $\stackrel{\mathcal{D}}{=}$ denotes equality in distribution. Furthermore, when $Q>0$ this recursion can be transformed into its additive equivalent, which corresponds to the maximum of a branching random walk and is also known as a high-order Lindley equation. We show that, under natural conditions, the asymptotic behavior of $R$ is power-law, i.e., $P(|R|>x)\sim Hx^{-\alpha}$, for some $\alpha>0$ and $H>0$. This has direct implications for the tail behavior of other well known branching recursions.

\vspace{5mm}



\end{abstract}

\section{Introduction}

In the recent years considerable attention \cite{Jel_Olv_10, Jel_Olv_12a, Jel_Olv_12b, Alsm_Mein_10a, Alsm_Mein_10b, Mirek_12, Alsm_Dam_Ment_13, Bur_Dam_Ment_Mir_13} has been given to the characterization and analysis of the solutions to the non homogeneous linear equation 
\begin{equation} \label{eq:IntroLinear}
R_{L} \stackrel{\mathcal{D}}{=} \sum_{i=1}^N C_i R_{L,i} + Q,
\end{equation}
where $(Q, N, C_1, C_2, \dots)$ is a real-valued random vector with $N \in \mathbb{N} \cup \{\infty\}$, $P\left(|Q| >0 \right)>0$, and $\{R_{L,i}\}_{i\in \mathbb{N}}$ is a sequence of i.i.d. random variables independent of $(Q, N, C_1, C_2, \dots)$ having the same distribution as $R_L$. Equation \eqref{eq:IntroLinear} has applications in a wide variety of fields, including the analysis of divide and conquer algorithms \cite{Ros_Rus_01, Nei_Rus_04}, e.g. Quicksort \cite{Fill_Jan_01}; the analysis of the PageRank algorithm \cite{Volk_Litv_08, Jel_Olv_10}; and kinetic gas theory \cite{Bur_Dam_Ment_Mir_13}. Our work in \cite{Jel_Olv_12a, Jel_Olv_12b} shows that the so-called endogenous solution, as termed in \cite{Aldo_Band_05}, of \eqref{eq:IntroLinear}, under the natural main root condition $E\left[ \sum_{i=1}^N |C_i|^\alpha \right] = 1$ with positive derivative $0 < E\left[ \sum_{i=1}^N |C_i|^\alpha \log |C_i| \right] < \infty$ for some $\alpha > 0$, has the power tail behavior, 
$$P(|R_L| > t) \sim H_L t^{-\alpha}, \qquad t \to \infty,$$
where $0 \leq H_L < \infty$. The main tool used in deriving this result was a generalization of Goldie's Implicit Renewal Theorem \cite{Goldie_91} to weighted branching trees. 

Motivated by a different set of applications, we study in this paper the maximum recursion on trees given by
\begin{equation} \label{eq:IntroMaximum}
R \stackrel{\mathcal{D}}{=} \left(\bigvee_{i=1}^N C_i R_i \right) \vee Q,
\end{equation}
where $(Q, N, C_1, C_2, \dots)$ is a random vector with $N \in \mathbb{N} \cup \{\infty\}$, nonnegative weights $\{C_i\}$, and $P\left(|Q| >0 \right)>0$, and $\{R_i\}_{i\in \mathbb{N}}$ is a sequence of i.i.d. random variables independent of $(Q, N, C_1, C_2, \dots)$ having the same distribution as $R$. Here and throughout the paper we use $x \vee y$ and $x \wedge y$ to denote the maximum and the minimum, respectively, of $x$ and $y$. We point out that by taking the logarithm in \eqref{eq:IntroMaximum} when $Q > 0$ a.s., we obtain the additive equivalent
\begin{equation} \label{eq:AdditiveMax}
X \stackrel{\mathcal{D}}{=} \bigvee_{i=1}^N \left( Y_i + X_i \right) \vee V,
\end{equation}
where $X = \log R$, $Y_i = \log C_i$, $V = \log Q$, and the $\{X_i\}_{i \in \mathbb{N}}$  are i.i.d. copies of $X$, independent of $(V, N, Y_1, Y_2, \dots)$. Note that for $N \equiv 1$ and $V \equiv 0$, \eqref{eq:AdditiveMax} reduces to the classical Lindley's equation, satisfied by the reflected random walk; and when $V \not\equiv 0$, the recursion corresponds to a random walk reflected on a random barrier. In general, the preceding additive equation has been studied in the literature of branching random walks since, when $V \equiv 0$, $X$ represents the range of the branching random walk (see \cite{Aldo_Band_05}, \S 4.2). Recursion \eqref{eq:AdditiveMax} was termed ``high-order Lindley equation" and studied in the context of queues with synchronization in \cite{Kar_Kel_Suh_94}. Unlike the classical Lindley equation, it was shown in \cite{Kar_Kel_Suh_94} that \eqref{eq:AdditiveMax} can have multiple solutions. A more complete analysis of the existence and the characterization of the entire family of solutions was carried out in \cite{Biggins_98} (e.g., see Theorem 1 in \cite{Biggins_98}). In addition, it can be shown that the study of \eqref{eq:AdditiveMax} arises in the context of today's massively parallel computing. More specifically, consider a job that is split into smaller pieces which are sent randomly to different processors, and these pieces need to communicate, i.e., need to be synchronized, in order to complete their processing.  In the limiting regime as the number of processors goes to infinity, a similar reasoning as in \cite{Kar_Kel_Suh_94} can be used to show that \eqref{eq:AdditiveMax} represents the delay for job completion in this massively parallel system. In addition to these applications, a better understanding of \eqref{eq:IntroMaximum} immediately leads to important insights to other max-plus branching recursions. More precisely, for the case of nonnegative weights, \eqref{eq:IntroMaximum} is a natural lower bound for many other recursions on trees \cite{Aldo_Band_05}, e.g., for the same set of weights, elementary arguments show that $R_L$ in \eqref{eq:IntroLinear} is stochastically larger than the solution $R$ in \eqref{eq:IntroMaximum}.

For all of the reasons described above, we study in this paper the tail behavior of the minimal/endogenous solution to the maximum recursion in \eqref{eq:IntroMaximum} (or \eqref{eq:AdditiveMax}). As shown in \cite{Biggins_98}, equation \eqref{eq:AdditiveMax} can have multiple solutions. It is worth noting that the minimal/endogenous solution we study here is also central in characterizing all other solutions, as stated in Theorem 1 of \cite{Biggins_98} (there $M$ is used to denote the minimal/endogenous solution). Furthermore, we would like to point out that under iterations of the fixed-point equation \eqref{eq:IntroMaximum} (or \eqref{eq:AdditiveMax}), the minimal/endogenous solution is the primary limiting value, unless one starts with very specific initial distributions (see Theorem 1(ii) in \cite{Biggins_98}); we will discuss this in more detail in Section \ref{S.MaxRec}.  In addition, we emphasize that the tail characterization of the other (non minimal solutions) was given in \cite{Biggins_98}, but the tail behavior of the minimal one was left open. 

Our first main result, stated in Theorem \ref{T.MaximumRecursion}, describes the tail behavior of the minimal/endogenous solution to the maximum recursion \eqref{eq:IntroMaximum} (or \eqref{eq:AdditiveMax}). In this regard, the application of the Implicit Renewal Theorem on Trees (see Theorem 3.4 \cite{Jel_Olv_12b}), under the natural conditions $E\left[ \sum_{i=1}^N C_i^\alpha \right] = 1$ and $0 < E\left[ \sum_{i=1}^N C_i^\alpha \log C_i \right] < \infty$ for some $\alpha > 0$, readily gives that
\begin{equation} \label{eq:MaxAsymptotics}
P(R > t) \sim H t^{-\alpha}, \qquad t \to \infty,
\end{equation}
where $0 \leq H < \infty$. However, the main difficulty in establishing the power-law behavior lies in proving that $H > 0$. 
Unlike in the linear case, it is not clear that this constant should be positive at all, since at first glance the expression which determines $H$ in Theorem 3.4 of Section \ref{S.MaxRec}, 
$$E\left[  (Q^+)^\alpha  \vee  \bigvee_{i=1}^N (C_i R_i^+)^\alpha   - \sum_{i=1}^N (C_i R_i^+ )^\alpha \right],$$
appears just as likely to be negative. Note also that a direct application of a ladder heights argument gives the positivity of the constant for the classical non-branching case ($N \equiv 1$), see Theorem 5.2 in \cite{Goldie_91}.  However, for the branching case no ladder heights equivalent is available. Hence, our first main contribution lies in a new sample-path construction showing that $H > 0$ under no additional assumptions (besides those needed for the application of Theorem~3.4 in \cite{Jel_Olv_12b}). Observe that in the additive case of equation \eqref{eq:AdditiveMax}, our result yields the exponential asymptotics $P(X > y) \sim H e^{-\alpha y}$, which is the generalization of the well known Cram\'er-Lundberg approximation. The latter is widely used in insurance risk theory and queueing.

Furthermore, as an immediate corollary one obtains the strict positivity of $H_L$ in the linear case with nonnegative $(Q, N, C_1, C_2, \dots)$.  In this setting, the work in \cite{Jel_Olv_12a} used a straightforward convexity argument to show that $H_L >0$ for $\alpha \geq 1$, but the corresponding question for $\alpha \in (0,1)$ was left open. The strict positivity of $H_L$ for $\alpha \in (0,1)$ was recently resolved in \cite{Alsm_Dam_Ment_13} as part of the more general real-valued case, but under additional assumptions that include $E\left[ \sum_{i=1}^N C_i^{\alpha+\epsilon} \right] < \infty$. Note that in the additive equation \eqref{eq:AdditiveMax}, this extra moment assumption corresponds to the finiteness of $\alpha+\epsilon$ exponential moments of the $\{Y_i\}$.  Since the new results on the maximum hold without such additional assumptions, Theorem \ref{T.MaximumRecursion} fully completes the prior work for nonnegative $(Q, N, C_1, C_2, \dots)$.  In addition, as already mentioned, the maximum is a natural lower bound for other max-plus recursions, and therefore this result can potentially be used to prove the power-tail asymptotics of the endogenous solutions to other recursions, e.g., the discounted tree sums considered in \cite{Aldo_Band_05}. 

We now go back to the linear recursion \eqref{eq:IntroLinear} with real-valued weights $(Q, C_1, C_2, \dots)$, which has recently been considered in  \cite{Alsm_Mein_10b, Jel_Olv_12b, Alsm_Dam_Ment_13} (see also \cite{Bur_Dam_Ment_Mir_13} for the multivariate case). The characterization of all the solutions of \eqref{eq:IntroLinear} when $Q$ is real-valued, $\{C_i \} \geq 0$ and $N < \infty$ a.s. was given in \cite{Alsm_Mein_10b}. The work in \cite{Jel_Olv_12b} establishes the Implicit Renewal Theorem on Trees for the real-valued case and shows that, under the usual conditions $E\left[ \sum_{i=1}^N |C_i|^\alpha \right] = 1$ and $0 < E\left[ \sum_{i=1}^N |C_i|^\alpha \log|C_i| \right] < \infty$, the endogenous solution to \eqref{eq:IntroLinear} has a power tail behavior of the form $H_L t^{-\alpha}$, $H_L \geq 0$. In that paper the strict positivity of $H_L$ in its full generality remained open.  It was this open problem that motivated the work in \cite{Alsm_Dam_Ment_13}, where it was shown, using complex analysis and analytical functions,  that $H_L > 0$ under the additional assumptions $N < \infty$ a.s., $E\left[ \sum_{i=1}^N |C_i|^{\alpha+\epsilon} \right] < \infty$ and $E\left[ \left( \sum_{i=1}^N |C_i| \right)^{\alpha+\epsilon} \right] < \infty$.

In this paper, we revisit the problem of the strict positivity of $H_L$ for the general real-valued case using our result on the maximum equation \eqref{eq:IntroMaximum} (with nonnegative weights $\{C_i\}$), under no additional assumptions on the vector $(N, C_1, C_2, \dots)$ besides those needed for Theorem 3.4 in \cite{Jel_Olv_12b}. However, we do require that $Q$ does not reduce to a constant given $(N, C_1, C_2, \dots)$. Our main set of arguments is based on L\'evy's symmetrization approach. We would like to mention that although the proof of Theorem 4.1 in \cite{Goldie_91} (the part that establishes the positivity of $H_L$ for the case $N \equiv 1$) also relies on symmetrization, our proof is completely different.  While in the one-dimensional case it was enough to center the weights around their median, in the branching case we need complete symmetry. More precisely, the proof consists in first showing the positivity of the constant for symmetric trees, and then extending it to the general case through a coupling argument; see Corollary~\ref{C.StrictPositivity}. In general, as previously stated, we expect that Theorem~\ref{T.MaximumRecursion} for the maximum, coupled with the Implicit Renewal Theorem on Trees (Theorem~3.4 in \cite{Jel_Olv_12b}), can be used to derive the exact power law asymptotics of the solutions to other branching recursions \cite{Aldo_Band_05}.

The paper is organized as follows. Section \ref{S.ModelDescription} includes a brief description of the weighted branching process. Section~\ref{S.MaxRec} contains our first main result about the asymptotic behavior of the minimal/endogenous solution to the maximum recursion \eqref{eq:IntroMaximum}, including the strict positivity of $H$. Section~\ref{S.LinearRecursion} presents our proof of the positivity of the constant $H_L$ for the general mixed-sign linear recursion \eqref{eq:IntroLinear}.

\section{Model description} \label{S.ModelDescription}

We use the model from \cite{Jel_Olv_12b} for defining a weighted branching tree.  First we construct a random tree $\mathcal{T}$. We use the notation $\emptyset$ to denote the root node of $\mathcal{T}$, and $A_n$, $n \geq 0$, to denote the set of all individuals in the $n$th generation of $\mathcal{T}$, $A_0 = \{\emptyset\}$. Let $Z_n$ be the number of individuals in the $n$th generation, that is, $Z_n = |A_n|$, where $| \cdot |$ denotes the cardinality of a set; in particular, $Z_0 = 1$. 

Next, let $\mathbb{N}_+ = \{1, 2, 3, \dots\}$ be the set of positive integers and let $U = \bigcup_{k=0}^\infty (\mathbb{N}_+)^k$ be the set of all finite sequences ${\bf i} = (i_1, i_2, \dots, i_n) \in U$, where by convention $\mathbb{N}_+^0 = \{ \emptyset\}$ contains the null sequence $\emptyset$. To ease the exposition, for a sequence ${\bf i} = (i_1, i_2, \dots, i_k) \in U$ we write ${\bf i}|n = (i_1, i_2, \dots, i_n)$, provided $k \geq n$, and  ${\bf i}|0 = \emptyset$ to denote the index truncation at level $n$, $n \geq 0$. Also, for ${\bf i} \in A_1$ we simply use the notation ${\bf i} = i_1$, that is, without the parenthesis. Similarly, for ${\bf i} = (i_1, \dots, i_n)$ we will use $({\bf i}, j) = (i_1,\dots, i_n, j)$ to denote the index concatenation operation, if ${\bf i} = \emptyset$, then $({\bf i}, j) = j$. 

We iteratively construct the tree as follows. Let $N$ be the number of individuals born to the root node $\emptyset$, $N_\emptyset = N$, and let $\{N_{\bf i} \}_{{\bf i} \in U, {\bf i} \neq \emptyset}$ be i.i.d. copies of $N$. Define now 
\begin{equation} \label{eq:AnDef}
A_1 = \{ i \in \mathbb{N}: 1 \leq i \leq N \}, \quad A_n = \{ ({\bf i}, i_n) \in U:  {\bf i} \in A_{n-1}, 1 \leq i_n \leq N_{\bf i} \}.
\end{equation}
It follows that the number of individuals $Z_n = |A_n|$ in the $n$th generation, $n \geq 1$, satisfies the branching recursion 
$$Z_{n} = \sum_{{\bf i } \in A_{n-1}} N_{\bf i}.$$ 

\begin{center}
\begin{figure}[h]
\begin{picture}(430,160)(0,0)
\put(0,0){\includegraphics[scale = 0.8, bb = 0 510 500 700, clip]{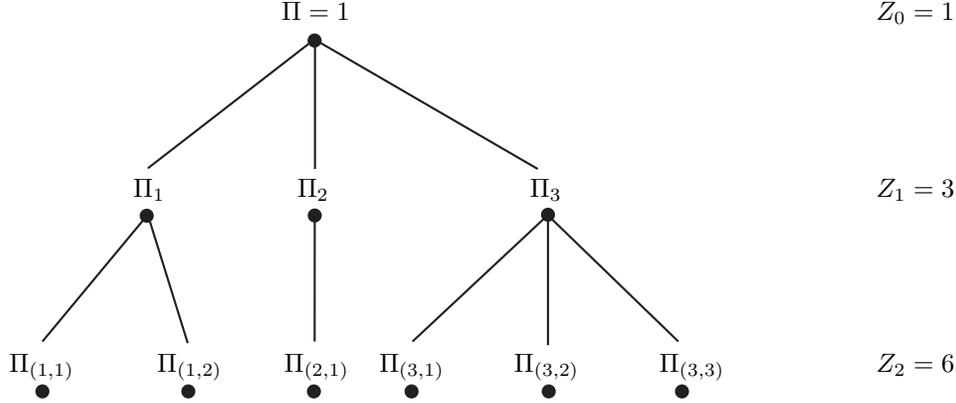}}
\put(125,150){\small $\Pi = 1$}
\put(69,83){\small $\Pi_{1}$}
\put(131,83){\small $\Pi_{2}$}
\put(219,83){\small $\Pi_{3}$}
\put(22,17){\small $\Pi_{(1,1)}$}
\put(78,17){\small $\Pi_{(1,2)}$}
\put(126,17){\small $\Pi_{(2,1)}$}
\put(162,17){\small $\Pi_{(3,1)}$}
\put(213,17){\small $\Pi_{(3,2)}$}
\put(268,17){\small $\Pi_{(3,3)}$}
\put(350,150){\small $Z_0 = 1$}
\put(350,83){\small $Z_1 = 3$}
\put(350,17){\small $Z_2 = 6$}
\end{picture}
\caption{Weighted branching tree}\label{F.Tree}
\end{figure}
\end{center}

Now, we construct the weighted branching tree $\mathcal{T}_{C}$ as follows. Let \linebreak $\{ (N_{\bf i}, C_{({\bf i}, 1)},  C_{({\bf i}, 2)}, \dots) \}_{{\bf i} \in U, {\bf i} \neq \emptyset}$ be a sequence of i.i.d. copies of $(N, C_1, C_2, \dots)$. $N_\emptyset$ determines the number of nodes in the first generation of $\mathcal{T}$ according to \eqref{eq:AnDef}, and each node in the first generation is then assigned its corresponding vector $(N_i, C_{(i,1)}, C_{(i,2)}, \dots)$ from the i.i.d. sequence defined above. In general, for $n \geq 2$, to each node ${\bf i} \in A_{n-1}$ we assign its corresponding $(N_{\bf i}, C_{({\bf i}, 1)}, C_{({\bf i}, 2)}, \dots )$ from the sequence and construct
$A_{n} = \{({\bf i}, i_{n}) \in U: {\bf i} \in A_{n-1}, 1 \leq i_{n} \leq N_{\bf i}\}$. 
For each node in $\mathcal{T}_{C}$ we also define the weight $\Pi_{(i_1,\dots,i_n)}$ via the recursion
$$ \Pi_{i_1} =C_{i_1}, \qquad \Pi_{(i_1,\dots,i_n)} = C_{(i_1,\dots, i_n)} \Pi_{(i_1,\dots,i_{n-1})}, \quad n \geq 2,$$
where $\Pi =1$ is the weight of the root node. Note that the weight $\Pi_{(i_1,\dots, i_n)}$ is equal to the product of all the weights $C_{(\cdot)}$ along the branch leading to node $(i_1, \dots, i_n)$, as depicted in Figure \ref{F.Tree}.

\section{The maximum recursion: $R = \left( \bigvee_{i=1}^N C_i R_i \right) \vee Q$} \label{S.MaxRec}

In this section, we study the maximum fixed-point equation given by
\begin{equation} \label{eq:Maximum}
R \stackrel{\mathcal{D}}{=} \left(\bigvee_{i=1}^N C_i R_i \right) \vee Q,
\end{equation}
where $(Q, N, C_1, C_2, \dots)$ is a random vector with $N \in \mathbb{N} \cup \{\infty\}$, $\{C_i\} \geq 0$ and $P(|Q| > 0) > 0$, and $\{R_i\}_{i\in \mathbb{N}}$ is a sequence of i.i.d. random variables independent of $(Q, N, C_1, C_2, \dots)$ having the same distribution as $R$.  As already mentioned, the additive version of \eqref{eq:Maximum}, given in \eqref{eq:AdditiveMax}, was termed ``high-order Lindley equation" and studied in the context of queues with synchronization in \cite{Kar_Kel_Suh_94}. The full characterization of its multiple solutions was given in \cite{Biggins_98}.  More recently, a related recursion where $Q \equiv 0$, $N = \infty$, and the $\{C_i\}$ are real valued deterministic constants, has been analyzed in \cite{Alsm_Rosl_08}.  The more closely related case of $Q \equiv 0$ and $\{C_i \} \geq 0$ being random was studied earlier in \cite{Jag_Ros_04}. For this and other max-plus equations appearing in a variety of applications see the survey by \cite{Aldo_Band_05}.

Using standard arguments, we start by constructing an endogenous solution to \eqref{eq:Maximum} on a tree and then we show that this solution is finite a.s. and unique under iterations provided that the initial values and the weights satisfy appropriate moment conditions. 

Following the notation of Section \ref{S.ModelDescription}, define the process
\begin{equation} \label{eq:V_k}
V_n = \bigvee_{{\bf i} \in A_n} Q_{{\bf i}} \Pi_{{\bf i}}, \qquad n \geq 0,
\end{equation}
on the weighted branching tree $\mathcal{T}_{Q, C}$. Recall that the convention is that  $(Q, N, C_1, C_2, \dots) = (Q_\emptyset, N_\emptyset, C_{(\emptyset, 1)}, C_{(\emptyset, 2)}, \dots)$ denotes the random vector corresponding to the root node. Next, define the process $\{R^{(n)}\}_{n \geq 0}$ according to
$$R^{(n)} = \bigvee _{k=0}^n V_k, \qquad n \geq 0.$$
It is not hard to see that $R^{(n)}$ satisfies the recursion
\begin{equation} \label{eq:MaxRecSamplePath}
R^{(n)} = \left( \bigvee_{j=1}^{N_\emptyset} C_{(\emptyset, j)} R_j^{(n-1)} \right) \vee Q_\emptyset = \left( \bigvee_{j=1}^{N} C_{j} R_j^{(n-1)} \right) \vee Q ,
\end{equation}
where $\{R_j^{(n-1)} \}$ are independent copies of $R^{(n-1)}$ corresponding to the tree starting with individual $j$ in the first generation and ending on the $n$th generation. One can also verify that
$$V_n = \bigvee_{k=1}^{N_{\emptyset}} C_{(\emptyset,k)} \bigvee_{(k,\dots, i_n) \in A_n} 
Q_{(k,\dots, i_n)} \prod_{j=2}^n C_{(k,\dots,i_j)}  \stackrel{\mathcal{D}}{=} \bigvee_{k=1}^N C_k V_{(n-1),k},$$
where $\{V_{(n-1),k}\}$ is a sequence of i.i.d. random variables independent of $(N, C_1, C_2, \dots)$ and having the same distribution as $V_{n-1}$. 

We now define the random variable $R$ according to
\begin{equation}
\label{eq:maxR}
R \triangleq \lim_{n\to \infty} R^{(n)} = \bigvee_{k=0}^\infty V_k.
\end{equation}

Note that $R^{(n)}$ is monotone increasing sample-pathwise, so $R$ is well defined. Also, by monotonicity of  $R^{(n)}$ and \eqref{eq:MaxRecSamplePath}, we obtain that $R$ solves
$$R = \left( \bigvee_{j=1}^{N_{\emptyset}} C_{(\emptyset,j)} R_j^{(\infty)} \right) \vee Q_\emptyset = \left( \bigvee_{j=1}^{N} C_{j} R_j^{(\infty)} \right) \vee Q,$$
where $\{R_j^{(\infty)} \}_{j \in \mathbb{N}}$ are i.i.d. copies of $R$, independent of $(Q, N, C_1, C_2, \dots)$, see also Section 2 in \cite{Biggins_98}. 
Clearly this implies that $R$, as defined by  \eqref{eq:maxR}, is a solution in distribution to \eqref{eq:Maximum}. However, this solution might be $\infty$. 
Next, we establish in the following lemma the finiteness of the moments of $R$, and in particular that $R < \infty$ a.s.; its proof uses standard contraction arguments but is included for completeness; e.g. see Theorem 6 (i) in \cite{Biggins_98}. Conditions under which $R$ is infinite a.s. can be found in Corollary 4 in \cite{Biggins_98}.

\begin{lemma} \label{L.Moments_R_Max}
Assume that $\rho_\beta = E\left[ \sum_{i=1}^N C_i^\beta \right]<1$ and 
$E[|Q|^\beta] < \infty$ for some $\beta>0$. Then, $E[|R|^\gamma] < \infty$ for all $0  < \gamma \leq \beta$, and in particular, $|R| < \infty$ a.s. 
Moreover, if $\beta  \geq 1$, $R^{(n)} \stackrel{L_\beta}{\to} R$, where $L_\beta$ stands for convergence in $(E|\cdot|^\beta)^{1/\beta}$ norm. 
\end{lemma}

\begin{proof}
Note that
\begin{align*}
|R|^\beta &= (R^+)^\beta + (R^-)^\beta = \bigvee_{k=0}^\infty (V_k^+)^\beta + \left( \left( \bigwedge_{k=0}^\infty (-V_k) \right)^+ \right)^\beta \\
&= \bigvee_{k=0}^\infty \bigvee_{{\bf i} \in A_k} (Q_{\bf i}^+)^\beta \Pi_{\bf i}^\beta +  \left( \left( \bigwedge_{k=0}^\infty \bigwedge_{{\bf i} \in A_k} (- Q_{\bf i}) \Pi_{\bf i}  \right)^+ \right)^\beta \\
&= \bigvee_{k=0}^\infty \bigvee_{{\bf i} \in A_k} (Q_{\bf i}^+)^\beta \Pi_{\bf i}^\beta +  \bigwedge_{k=0}^\infty \bigwedge_{{\bf i} \in A_k} (Q_{\bf i}^-)^\beta \Pi_{\bf i}^\beta \\
&\leq \sum_{k=0}^\infty \sum_{{\bf i} \in A_k} |Q_{\bf i}|^\beta \Pi_{\bf i}^\beta.
\end{align*}
It follows that
\begin{align*}
E\left[ |R|^\beta \right] &\leq E\left[ \sum_{k=0}^\infty \sum_{{\bf i} \in A_k} |Q_{\bf i}|^\beta \Pi_{\bf i}^\beta \right] = \sum_{k=0}^\infty E\left[ \sum_{{\bf i} \in A_k} |Q_{\bf i}|^\beta \Pi_{\bf i}^\beta \right] \\
&= \sum_{k=0}^\infty E[|Q|^\beta] \rho_\beta^k = \frac{E[|Q|^\beta]}{1-\rho_\beta} < \infty,
\end{align*}
That $R^{(n)} \stackrel{L_\beta}{\to} R$ whenever $\beta\geq 1$ follows from noting that $|R^{(n)} - R|^\beta \leq \left| \bigvee_{k = n+1}^\infty V_k \right|^\beta$ and the same arguments used above.
\end{proof}

Although this paper focuses only on the solution $R$ defined by \eqref{eq:maxR}, it is important to mention that equation \eqref{eq:Maximum} can have multiple solutions, as the work in \cite{Biggins_98} describes. The solution $R$ receives the name ``endogenous" since it is constructed explicitly from the weighted branching tree, and the name ``minimal" since it is the stochastically smallest solution, in the sense that any other solution $R'$ to \eqref{eq:Maximum} satisfies $P(R' > t) \leq P(R > t)$ for all $t > 0$. For the case when $Q \geq 0$ and there exists a unique $\upsilon > 0$ such that $E\left[ \sum_{i=1}^N C_i^\upsilon \right] = 1$ and $- \infty < E\left[ \sum_{i=1}^N C_i^\upsilon \log C_i \right] < 0$ (referred to as the ``regular case"), Theorem 1 (ii) and (iii) of \cite{Biggins_98} characterizes the entire family of solutions to \eqref{eq:Maximum}. Moreover, under some additional technical conditions, all other solutions to \eqref{eq:Maximum} are given in terms of $R$ ($M = \log R$ in \cite{Biggins_98}) and the limit $W(\upsilon)$ of the martingale $W_k(\upsilon) = \sum_{{\bf i} \in A_k} \Pi_{\bf i}^\upsilon$.  To better understand the nature of these other solutions, as well as to highlight the importance of the endogenous/minimal solution $R$, we will next define the process $\{R_n^*\}$ that is obtained from iterating equation \eqref{eq:Maximum} starting from an initial value $R_0^*$. 

Let
\begin{equation*}
R_n^* \triangleq R^{(n-1)} \vee V_n(R_0^*), \qquad n \geq 1,
\end{equation*}
where
\begin{equation} \label{eq:MaxLastWeights}
V_n(R_0^*) = \bigvee_{{\bf i} \in A_n} R^*_{0,{\bf i}} \Pi_{{\bf i}},
\end{equation}
and $\{ R_{0,{\bf i}}^*\}_{{\bf i} \in U}$ are i.i.d. copies of an initial value $R_0^*$, independent of the entire weighted tree $\mathcal{T}_{Q,C}$. $R_0^*$ is referred to as the ``terminal" value in \cite{Biggins_98} ($T = \log R_0^*$, $R_0^* \geq 0$) since it corresponds to the value of the leaves in the weighted branching tree with finitely many generations. It follows from \eqref{eq:MaxRecSamplePath} and \eqref{eq:MaxLastWeights} that
\begin{equation*} 
R_{n+1}^* =  \bigvee_{j=1}^{N} C_j \left( R_{j}^{(n-1)}  \vee \bigvee_{{\bf i} \in A_{n,j}} R_{0,{\bf i}}^* \prod_{k=2}^n C_{(j,\dots,i_k)} \right) \vee Q = \bigvee_{j=1}^N C_j R_{n,j}^* \vee Q,
\end{equation*}
where $\{ R_{j}^{(n-1)} \}$ are independent copies of $R^{(n-1)}$ corresponding to the tree starting with individual $j$ in the first generation and ending on the $n$th generation, and $A_{n,j}$ is the set of all nodes in the $(n+1)$th generation that are descendants of individual $j$ in the first generation. Moreover, $\{R_{n,j}^*\}$ are i.i.d. copies of $R_n^*$, and thus, $R_n^*$ is equal in distribution to the process obtained by iterating \eqref{eq:Maximum} with an initial condition $R_0^*$. This process can be shown to converge in distribution to $R$ for any initial condition $R_0^*$ satisfying the following moment condition (see also Theorem 9 in \cite{Biggins_98}). 

\begin{lemma} \label{L.ConvergenceMax}
Suppose $E[|Q|^\beta], E[|R_0^*|^\beta] < \infty$ and $\rho_\beta < 1$ for some $\beta >0$, then
$$R_n^* \Rightarrow R,$$
with $E[|R|^\beta] < \infty$. Furthermore, under these assumptions, the distribution of $R$ is the unique solution with finite $\beta$-moment to recursion \eqref{eq:Maximum}. 
\end{lemma}

\begin{proof}
The result will follow from Slutsky's Theorem (see Theorem 25.4, p. 332 in \cite{Billingsley_1995}) once we show that $V_n(R_0^*) \Rightarrow 0$. To this end, recall that $V_n(R_0^*)$ is the same as $V_n$ if we substitute the $Q_{{\bf i}}$ by the $R_{0,{\bf i}}^*$. Then, for every $\epsilon > 0$ we have that
\begin{align*}
P( |V_n(R_0^*)| > \epsilon) &\leq \epsilon^{-\beta} E[ |V_n(R_0^*)|^\beta] \leq \epsilon^{-\beta} E\left[ \sum_{{\bf i} \in A_n} |R_{0,{\bf i}}^*| \Pi_{\bf i}^\beta \right] = \epsilon^{-\beta} \rho_\beta^n E[|R_0^*|^\beta]  .
\end{align*}
Since by assumption the right-hand side converges to zero as $n \to \infty$, then $R_n^* \Rightarrow R$. Furthermore, $E[|R|^\beta] < \infty$ by Lemma \ref{L.Moments_R_Max}. Clearly, under the assumptions, the distribution of $R$ represents the unique solution to \eqref{eq:Maximum}, since any other possible solution with finite $\beta$-moment would have to converge to the same limit.  
\end{proof}

\begin{remarks}
(a) Lemma \ref{L.ConvergenceMax} establishes a certain type of uniqueness of the solution to \eqref{eq:Maximum}, in the sense that $R$ is the only possible limit for the iterative process $\{R_n^*\}$ for any initial value $R_0^*$ possessing finite $\beta$ moment. It is therefore to be expected that all other solutions to the maximum recursion must arise from violating this assumption.  (b) Theorem 1 (ii) of \cite{Biggins_98} states that in the regular case (see the comments after Lemma \ref{L.Moments_R_Max}), if $R_0^*\geq 0$ and $\lim_{t \to \infty} t^{\upsilon } P(R_0^* > t) = \gamma$ ($\upsilon < \alpha$), then $R_n^* \Rightarrow R(\gamma)$, where 
$$P( R(\gamma) \leq t) = E\left[ 1(R \leq t) e^{-\gamma W(\upsilon) t^{\upsilon}} \right].$$
Moreover, $R(\gamma)$ solves \eqref{eq:Maximum} provided $R < \infty$ a.s. and $E\left[ W_1(\upsilon) \log^+ W_1(\upsilon) \right] < \infty$. 
\end{remarks}

Now we are ready to state the main result of this section, which characterizes the asymptotic behavior of $R$. 

\begin{theo} \label{T.MaximumRecursion}
Let $(Q, N, C_1, C_2, \dots)$ be a random vector with $N \in \mathbb{N} \cup \{\infty\}$, $\{C_i\} \geq 0$ and $P(|Q| > 0) > 0$, and $R$ be the solution to \eqref{eq:Maximum} given by  \eqref{eq:maxR}. 
Suppose that there exists $j \geq 1$ with $P(N\ge j, C_j>0)>0$ such that the measure $P(\log C_j \in du, C_j > 0, N\ge j)$ is nonarithmetic, and 
that for some $\alpha > 0$,  $E[|Q|^\alpha] < \infty$, $0 < E \left[ \sum_{i=1}^N C_i^\alpha \log C_i \right] < \infty$ and  $ E \left[ \sum_{i=1}^N C_i^\alpha \right] = 1$. In addition, assume
\begin{enumerate}
\item $E\left[ \left( \sum_{i=1}^N C_i \right)^\alpha \right] < \infty$, if $\alpha > 1$; or,
\item $E\left[ \left( \sum_{i=1}^N C_i^{\alpha/(1+\epsilon)}\right)^{1+\epsilon} \right] < \infty$ for some $0 < \epsilon< 1$, if $0 < \alpha \leq 1$.
\end{enumerate}
Then,
$$P(R > t) \sim H t^{-\alpha}, \qquad P(R < -t) = o\left( t^{-\alpha} \right), \qquad t \to \infty,$$
where $0 \leq H < \infty$ is given by
\begin{align*}
H &=  \frac{1}{E\left[ \sum_{i=1}^N C_i^\alpha \log C_i \right]} \int_0^\infty v^{\alpha-1} \left( P(R > v) - E\left[ \sum_{i=1}^N 1(C_i R > v) \right] \right) dv \\
&= \frac{E\left[  (Q^+)^\alpha  \vee  \bigvee_{i=1}^N (C_i R_i^+)^\alpha   - \sum_{i=1}^N (C_i R_i^+ )^\alpha \right]}{\alpha E\left[ \sum_{i=1}^N C_i^\alpha \log C_i \right]}.
\end{align*}
Furthermore, $H > 0$ if and only if $P(Q^+ > 0) > 0$. 
\end{theo}

\begin{remarks}
(a) The condition $E[|Q|^\alpha] < \infty$ is only needed to obtain the result about the negative tail. The result about the positive tail $P(R > t)$ only requires $E[ (Q^+)^\alpha ] < \infty$. (b) The equivalent result for the lattice case can be obtained by using the corresponding Implicit Renewal Theorem on Trees in \cite{Jel_Olv_12b}. (c) Corollary 5 in \cite{Biggins_98} provides upper bounds for the tail behavior of any finite solution to the maximum equation \eqref{eq:Maximum}. 
\end{remarks}

\begin{proof}
The first part of the proof about the right tail, $P(R > t)$, will follow from an application of the Implicit Renewal Theorem on Trees, Theorem 3.4 in \cite{Jel_Olv_12b}, once we verify the finiteness of
\begin{equation} \label{eq:integrability}
\int_0^\infty \left| P(R > t) - E\left[ \sum_{i=1}^N 1(C_i R_i > t) \right] \right| t^{\alpha-1} dt.
\end{equation}
To see that \eqref{eq:integrability} is indeed finite, note that by Lemma 4.10 in \cite{Jel_Olv_12b} we have that
$$0 \leq \int_0^\infty \left( E\left[ \sum_{i=1}^N 1(C_i R_i > t) \right]  - P\left( \bigvee_{i=1}^N C_i R_i > t \right) \right) t^{\alpha-1} dt < \infty.$$
Also, since $R^* \triangleq \left( \bigvee_{i=1}^N C_i R_i \right) \vee Q  \geq \bigvee_{i=1}^N C_i R_i$, then
\begin{align*}
0 &\leq \int_0^\infty \left( P(R^* > t) - P\left( \bigvee_{i=1}^N C_i R_i > t \right) \right) t^{\alpha-1} dt \\
&= \frac{1}{\alpha} E\left[ \left( \left( \left( \bigvee_{i=1}^N C_i R_i \right) \vee Q \right)^+ \right)^\alpha -  \left(\left( \bigvee_{i=1}^N C_i R_i \right)^+ \right)^\alpha \right] \\
&= \frac{1}{\alpha} E\left[ (Q^+)^\alpha \vee \bigvee_{i=1}^N (C_i R_i^+)^\alpha   -  \bigvee_{i=1}^N (C_i R_i^+)^\alpha  \right] \\
&\leq \frac{1}{\alpha} E[ (Q^+)^\alpha ].
\end{align*}
Combining these two observations gives that \eqref{eq:integrability} is finite, and by Theorem~3.4 (a) in \cite{Jel_Olv_12b} we obtain the result with the integral representation of $H$. To derive the second expression for $H$ follow the same steps used at the end of the proof of Theorem~4.1 in \cite{Jel_Olv_12a}.

For the negative tail, $P(R < -t)$, simply note that
\begin{align*}
P(R < -t) &= P\left( \left( \bigvee_{i=1}^N C_i R_i \right) \vee Q < -t \right) \leq P(Q < -t) \\
&\leq P( |Q| > t ) \leq E\left[ |Q|^\alpha 1(|Q|^\alpha > t) \right] t^{-\alpha},
\end{align*}
where in the last step we used Markov's inequality. Since $E[|Q|^\alpha ] < \infty$, then $E[ |Q|^\alpha 1(|Q|^\alpha > t) ] = o(1)$ as $t \to \infty$, proving the result. 

The rest of the proof is devoted to showing that the constant $H > 0$ if and only if $P(Q^+ > 0) > 0$. Note that if $Q \leq 0$ a.s. then $R^+ = 0$ and therefore $H = 0$, so it only remains to show that $H > 0$ whenever $P(Q^+ > 0) > 0$. Hence, assume from now on that $P(Q^+ > 0) > 0$. 

The main idea of the proof is to construct a minorizing random variable for $R$ for which we can directly estimate the expectation appearing in the numerator of $H$. We start by fixing $0 < \delta < E[ (Q^+)^\alpha ] \wedge 1$ and choosing $\alpha/2 < \beta < \alpha$ and $q > 0$ such that $\rho_\beta < 1$, $E\left[ \left( \sum_{i=1}^N C_i^\beta \right)^{\alpha/\beta} \right] < \infty$, and $E[ (Q^+)^\alpha 1(Q^+ > q) ] < \delta/6$; define $K = \beta^{-1} \int_0^\infty \left( e^{-u} - 1 + u \right) u^{-\alpha/\beta-1} du < \infty$. Note that such $\beta$ always exists under the assumptions of the theorem, since when $0 < \alpha \leq 1$ we have that for any $\alpha/(1+\epsilon) \leq \beta < \alpha$,
$$E\left[ \left( \sum_{i=1}^N C_i^\beta \right)^{\alpha/\beta} \right] \leq E\left[ \left( \sum_{i=1}^N C_i^{\frac{\alpha}{1+\epsilon} } \right)^{\frac{\alpha}{\beta} \cdot \frac{(1+\epsilon) \beta}{\alpha}} \right] = E\left[ \left( \sum_{i=1}^N C_i^{\frac{\alpha}{1+\epsilon} } \right)^{1+\epsilon} \right] < \infty,$$
and when $\alpha > 1$ we have that for any $1 \leq \beta < \alpha$,
$$E\left[ \left( \sum_{i=1}^N C_i^\beta \right)^{\alpha/\beta} \right] \leq E\left[ \left( \sum_{i=1}^N C_i \right)^{\frac{\alpha}{\beta} \cdot \beta} \right] = E\left[ \left( \sum_{i=1}^N C_i \right)^{\alpha} \right] < \infty.$$
Now let $\{D_{{\bf i}, j} : {\bf i} \in U, 1 \leq j \leq r \}$ be nonnegative i.i.d. random variables, independent of $\mathcal{T}$, having the same distribution as $D$, where $D$ satisfies
$$0 \leq D \leq d \, \text{ a.s.}, \qquad E[ D^\alpha ] = 1 \qquad \text{and} \qquad E[ D^\beta ] < 1,$$
(e.g., take $D$ to have density $f(x) = (\alpha/2) x^{\alpha-1} 1(0 \leq x \leq 2^{1/\alpha})$).
For each ${\bf i} \in \mathcal{T}$ define the random variable 
$$\mathcal{Q}_{\bf i} = Q_{\bf i} \prod_{j=1}^r D_{({\bf i}, j)},$$
where $r \in \mathbb{N}$ is such that 
$$\frac{q^\alpha E[(Q^+)^\beta]}{(\delta/6)(1-\rho_\beta)} (E[D^\beta])^r < \delta/6$$
and
$$K E\left[ \left( \sum_{i=1}^N C_i^\beta \right)^{\alpha/\beta} \right] \left( \frac{E[ (Q^+)^\beta]}{1-\rho_\beta} \right)^{\alpha/\beta} (E[D^\beta])^r < \delta/2.$$
Let 
$$\mathcal{R} = \bigvee_{{\bf i} \in \mathcal{T}} \mathcal{Q}_{\bf i} \Pi_{\bf i},$$
and note that for any $t > 0$,
\begin{align*}
P( R > t) &= P(R^+ > t) = P\left( \left( \bigvee_{{\bf i} \in \mathcal{T}} Q_{\bf i} \Pi_{\bf i} \right)^+ > t \right) = P\left( \bigvee_{{\bf i} \in \mathcal{T}} d^r Q_{\bf i}^+ \Pi_{\bf i} > d^r t \right) \\
&\geq P\left( \bigvee_{{\bf i} \in \mathcal{T}} \prod_{j=1}^r D_{({\bf i},j)} Q_{\bf i}^+ \Pi_{\bf i} > d^r t \right) = P\left( \bigvee_{{\bf i} \in \mathcal{T}} \mathcal{Q}_{\bf i}^+ \Pi_{\bf i} > d^r t \right) \\
&= P(\mathcal{R} > d^r t ).
\end{align*}
We now apply the first part of this theorem to the new random variable $\mathcal{R}$ to obtain 
$$P(\mathcal{R} > v ) \sim \frac{E\left[ (\mathcal{Q}^+)^\alpha \vee \bigvee_{i=1}^N (C_i \mathcal{R}_i^+ )^\alpha  - \sum_{i=1}^N (C_i \mathcal{R}_i^+ )^\alpha \right]}{\alpha E\left[ \sum_{i=1}^N C_i^\alpha \log C_i \right]} \cdot v^{-\alpha}$$
as $v \to \infty$. The positivity of $H$ will then follow once we show
$$\mathcal{E} \triangleq E\left[(\mathcal{Q}^+)^\alpha \vee \bigvee_{i=1}^N (C_i \mathcal{R}_i^+ )^\alpha  - \sum_{i=1}^N (C_i \mathcal{R}_i^+ )^\alpha \right] > 0.$$
We start by writing $\mathcal{E}$ as
\begin{align*}
\mathcal{E} &= E\left[ \left(  (\mathcal{Q}^+)^\alpha  - \bigvee_{i=1}^N (C_i \mathcal{R}_i^+ )^\alpha \right)^+ \right]  - E\left[  \sum_{i=1}^N (C_i \mathcal{R}_i^+ )^\alpha - \bigvee_{i=1}^N (C_i \mathcal{R}_i^+ )^\alpha \right]  \\
&\triangleq \mathcal{E}_1 - \mathcal{E}_2.
\end{align*}

To analyze $\mathcal{E}_1$ note that
\begin{align*}
\mathcal{E}_1 &\geq E\left[ \left(  (\mathcal{Q}^+)^\alpha 1(Q^+ \leq q)  - \bigvee_{i=1}^N (C_i \mathcal{R}_i^+ )^\alpha \right)^+ 1\left( \bigvee_{i=1}^N (C_i \mathcal{R}_i^+)^\alpha \leq \delta/6 \right) \right] \\
&\geq E\left[ \left( (\mathcal{Q}^+)^\alpha 1(Q^+ \leq q) - \delta/6 \right)^+ \right] \\
&\hspace{5mm}- E\left[ \left( (\mathcal{Q}^+)^\alpha 1(Q^+ \leq q) - \delta/6 \right)^+ 1\left( \bigvee_{i=1}^N (C_i \mathcal{R}_i^+)^\alpha > \delta/6 \right) \right] \\
&\geq E\left[ ((\mathcal{Q}^+)^\alpha 1(Q^+ \leq q) \right] - \frac{\delta}{6} -  q^\alpha E\left[ \left( \prod_{j=1}^r D_j^\alpha \right) 1\left( \bigvee_{i=1}^N (C_i \mathcal{R}_i^+)^\alpha > \delta/6 \right) \right] \\
&= E\left[ (Q^+)^\alpha \right] - E\left[ (Q^+)^\alpha 1(Q^+ > q) \right] - \frac{\delta}{6} - q^\alpha P\left( \bigvee_{i=1}^N (C_i \mathcal{R}_i^+)^\alpha > \delta/6 \right),
\end{align*}
where in the last equality we used the observation that $(\mathcal{Q}^+)^\alpha = (Q^+)^\alpha \prod_{j=1}^r D_j^\alpha$, where $\prod_{j=1}^r D_j^\alpha$ is independent of $\mathcal{T}$, and $E\left[ \prod_{j=1}^r D_j^\alpha \right] = 1$. It follows that
\begin{align*}
\mathcal{E}_1 &\geq E[(Q^+)^\alpha] - \frac{\delta}{3} - q^\alpha P\left( \bigvee_{i=1}^N (C_i \mathcal{R}_i^+ )^\beta > (\delta/6)^{\beta/\alpha} \right) \\
&\geq E[(Q^+)^\alpha] - \frac{\delta}{3} - \frac{q^\alpha}{(\delta/6)^{\beta/\alpha}} E\left[ \sum_{i=1}^N (C_i \mathcal{R}_i^+)^\beta \right] \qquad \text{(by Markov's inequality)} \\
&= E[(Q^+)^\alpha] - \frac{\delta}{3} - \frac{q^\alpha \rho_\beta}{(\delta/6)^{\beta/\alpha}} E\left[ (\mathcal{R}^+)^\beta \right].
\end{align*}
By the same arguments used in the proof of Lemma \ref{L.Moments_R_Max}, 
\begin{equation} \label{eq:BoundExpR}
E[ (\mathcal{R}^+)^\beta ] \leq E\left[ \sum_{k=0}^\infty \sum_{{\bf i} \in A_k} (\mathcal{Q}_{\bf i}^+)^\beta \Pi_{\bf i}^\beta \right] = \frac{E\left[ (\mathcal{Q}^+)^\beta \right]}{1-\rho_\beta} = \frac{ \left( E[D^\beta] \right)^r E[ (Q^+)^\beta]}{1-\rho_\beta}.
\end{equation}
Our choice of $r$ now guarantees that
$$\mathcal{E}_1 \geq E[ (Q^+)^\alpha ] - \frac{\delta}{3} - \frac{q^\alpha E[(Q^+)^\beta]}{(\delta/6) (1-\rho_\beta)} (E[D^\beta])^r > E[(Q^+)^\alpha] - \frac{\delta}{2}.$$

It remains to bound $\mathcal{E}_2$. Follow the same steps as in the proof of Lemma~4.6 in \cite{Jel_Olv_12a} to obtain
\begin{align*}
\mathcal{E}_2 &= \int_0^\infty E\left[ \sum_{i=1}^N 1(C_i \mathcal{R}_i^+ > t) - 1\left( \bigvee_{i=1}^N C_i \mathcal{R}_i^+ > t \right) \right] t^{\alpha-1} dt \\
&\leq E\left[ \beta^{-1} \left( E[ (\mathcal{R}^+)^\beta] \sum_{i=1}^N C_i^\beta \right)^{\alpha/\beta} \int_0^\infty \left( e^{-u} -1 + u \right) u^{-\alpha/\beta-1} du \right] \\
&\leq K E\left[ \left( \sum_{i=1}^N C_i^\beta \right)^{\alpha/\beta} \right] \left( \frac{ \left( E[D^\beta] \right)^r E[ (Q^+)^\beta]}{1-\rho_\beta} \right)^{\alpha/\beta} \qquad \text{(by \eqref{eq:BoundExpR})}.
\end{align*}
Our choice of $r$ now gives $\mathcal{E}_2 < \delta/2$. We conclude that
$$\mathcal{E} > E[(Q^+)^\alpha] - \delta > 0.$$
\end{proof}

\section{The linear recursion} \label{S.LinearRecursion}

In this section of the paper we explain how Theorem~\ref{T.MaximumRecursion}, which establishes the power-law behavior of the endogenous solution $R$ to the maximum equation \eqref{eq:Maximum}, can be used to show that the constant $H_L$ given by Theorem~4.6 in \cite{Jel_Olv_12b} is strictly positive. 

Consider the linear equation
\begin{equation} \label{eq:Linear}
R_L \stackrel{\mathcal{D}}{=} \sum_{i=1}^N C_i R_{L,i} + Q,
\end{equation}
where $(Q, N, C_1, C_2, \dots)$ is a real-valued random vector with $N \in \mathbb{N} \cup \{\infty\}$ and $P(|Q|>0) > 0$, and $\{R_{L,i}\}_{i \in \mathbb{N}}$ is a sequence of i.i.d. random variables independent of $(Q, N, C_1, C_2, \dots)$ having the same distribution as $R_L$. The results in this section refer to the endogenous solution given by
\begin{equation} \label{eq:linearR}
R_L = \sum_{k=0}^\infty \sum_{{\bf i} \in A_k} \Pi_{\bf i} Q_{\bf i}.
\end{equation}
We refer the reader to Section 4 in \cite{Jel_Olv_12b} for detailed conditions under which $R_L$ is well defined and how it solves \eqref{eq:Linear}. We point out that all our results for the maximum recursion assume that the weights $\{C_i\}$ are nonnegative, so the connection between $R_L$ and $R$ is much more difficult to make in this case. 

As mentioned in the introduction, the idea behind our proof lies in first considering what we call a ``symmetric tree", and using a novel argument to show that the corresponding endogenous solution to the linear recursion on this symmetric tree follows a power-law asymptotic behavior with a strictly positive constant of proportionality. The second step is to construct a symmetric tree using two coupled versions of a general (non-symmetric) tree and show how the solution to the general case is lower bounded by the symmetric solution. The first of these two steps is given in the following proposition.

\begin{prop} \label{P.Symmetry}
Let $(Q,N, C_1, C_2, \dots)$ be a random vector with $N \in \mathbb{N} \cup \{\infty\}$ and $P(|Q| > 0) > 0$, and $R_L$ be the solution to \eqref{eq:Linear} given by  \eqref{eq:linearR}. Assume that for some $0 < \beta \leq 1$, $E[|Q|^\beta] < \infty$ and $E\left[ \sum_{j=1}^N |C_j|^\beta \right] < 1$. In addition, suppose that
\begin{equation} \label{eq:symmetricVector}
(Q,N, C_1, C_2, \dots) \stackrel{\mathcal{D}}{=} (-Q,N, C_1, C_2, \dots).
\end{equation}
Then,
$$P( |R_L| > t) \geq \frac{1}{2} P\left( \max_{{\bf i} \in \mathcal{T}} |\Pi_{\bf i} Q_{\bf i}| > t \right).$$
\end{prop}

\begin{remarks}
(a) We call a weighted branching tree whose root vector satisfies \eqref{eq:symmetricVector} a symmetric tree. We will show how one can easily construct such trees in the proof of Corollary \ref{C.StrictPositivity} below. (b) That the solution $R_L$ for symmetric trees follows a power-law behavior with strictly positive constant of proportionality if and only if $Q \not\equiv 0$ immediately follows from the preceding proposition and Theorem \ref{T.MaximumRecursion} applied to the weighted branching tree having root vector $(|Q|, N, |C_1|, |C_2|, \dots)$. (c) The proof of Proposition~\ref{P.Symmetry} follows the ideas of the L\'evy-type maximal inequalities from \cite{Pena_Gine_1999} (see Theorem 1.1.1) adapted to weighted branching trees.
\end{remarks}

\begin{proof}[Proof of Proposition \ref{P.Symmetry}]
We start by defining the process
$$W_0 = Q, \qquad W_k =  \sum_{{\bf i} \in A_k} \Pi_{\bf i} Q_{\bf i}, \quad k \in \mathbb{N},$$
and with some abuse of notation, the process
$$V_0 = |Q|, \qquad V_k = \bigvee_{{\bf i} \in A_k} |\Pi_{\bf i} Q_{\bf i}|, \quad k \in \mathbb{N}.$$
Next, consider the events
$$B_0 = \{ V_0 > t \}, \qquad B_k = \left\{ \max_{0 \leq i \leq k-1} V_i \leq t, \, V_k > t \right\}, \quad k \in \mathbb{N},$$
and note that they are disjoint and satisfy
$$P\left( \max_{k \geq 0} V_k > t \right) = \sum_{k=0}^\infty P(B_k) = \sum_{k=0}^\infty \left( P\left( B_k, |R_L| > t \right) + P\left( B_k, |R_L| \leq t \right) \right).$$
To analyze the second probability on the right hand side, for each $k = 1,2, 3, \dots$, let $m_k: \mathbb{N} \to \mathbb{N}^k$ be a bijective function, and use it to define the events $B_{k,1} = B_k \cap \left\{ |\Pi_{m_k(1)} Q_{m_k(1)}| > t \right\}$ and
$$B_{k,j} = B_k \cap \left\{ \max_{1 \leq r \leq j-1} |\Pi_{m_k(r)} Q_{m_k(r)}| \leq t, \, |\Pi_{m_k(j)} Q_{m_k(j)}| > t \right\}, \quad j = 2, 3, 4, \dots,$$
where the convention is to set $\Pi_{m_k(r)} \equiv 0$ if $m_k(r) \notin A_k$. Note that the $\{B_{k,j}\}$ are disjoint and $P(B_k) = \sum_{j=1}^\infty P(B_{k,j})$. The key observation is that under the symmetry assumptions of the lemma we have that for any $k \geq 0$, and $r \in \mathbb{N}$ the sequences
$$\{ (Q_{\bf i}, N_{\bf i}, C_{({\bf i},1)}, C_{({\bf i},2)}, \dots): {\bf i} \in \mathcal{T} \}$$
and 
$$\{ (Q_{\bf i}, N_{\bf i}, C_{({\bf i},1)}, C_{({\bf i},2)}, \dots):  {\bf i} \in \mathcal{T}, {\bf i} \neq m_k(r) \} \cup \{ (-Q_{m_k(r)}, N_{m_k(r)}, C_{(m_k(r),1)}, C_{(m_k(r),2)}, \dots) \}$$
have the same distribution. It follows that for any $k \geq 0$ and $r \in \mathbb{N}$,
\begin{align*}
R_L &= \sum_{j \neq k} W_j + \sum_{{\bf i} \in A_k, {\bf i} \neq m_k(r)} \Pi_{\bf i} Q_{\bf i} + \Pi_{m_k(r)} Q_{m_k(r)}  \\
&\stackrel{\mathcal{D}}{=}  \sum_{ j \neq k} W_j + \sum_{{\bf i} \in A_k, {\bf i} \neq m_k(r)} \Pi_{\bf i} Q_{\bf i} - \Pi_{m_k(r)} Q_{m_k(r)} \\
&= R_L - 2 \Pi_{m_k(r)} Q_{m_k(r)},
\end{align*}
and since the events $\{B_{k,j}\}$ are insensitive to changes in the sign of the $\{Q_{\bf i}\}$, we have that for any $k \geq 0$, 
\begin{align*}
P\left( B_k, |R_L| \leq t \right) &= \sum_{r=1}^\infty P\left( B_{k,r}, |R_L| \leq t \right) \\
&= \sum_{r=1}^\infty P\left( B_{k,r}, \, \left| 2\Pi_{m_k(r)} Q_{m_k(r)} - R_L  \right| \leq t \right) \\
&\leq \sum_{r=1}^\infty P\left( B_{k,r}, \, 2\left| \Pi_{m_k(r)} Q_{m_k(r)} \right| - \left|R_L \right| \leq t \right) \\
&= \sum_{r=1}^\infty P\left( B_{k,r}, \,   \left|R_L \right| \geq 2 \left| \Pi_{m_k(r)} Q_{m_k(r)} \right| -t   \right) \\
&\leq \sum_{r=1}^\infty P\left( B_{k,r}, \,   \left|R_L \right| > t   \right) \qquad \text{(since $\left| \Pi_{m_k(r)} Q_{m_k(r)} \right| > t$ on $B_{k,r}$)} \\
&= P\left( B_k, \left| R_L \right| > t \right).
\end{align*}
This completes the proof. 
\end{proof}

We now proceed to the second step of our proof, the one that shows how to lower bound the endogenous solution $R_L$ in the general case with the solution to the linear equation \eqref{eq:Linear} on a closely related symmetric tree. The following technical lemma will be useful to explain our construction.

\begin{lemma} \label{L.Construction}
Let $X$ be a real-valued random variable and let ${\bf Y} = (Y_1, Y_2, \dots) \in \mathbb{R}^\infty$ be a random vector on the same probability space. Define for $x \in \mathbb{R}$,
$$F_{\bf Y}(x) = E\left[ 1( X \leq x) | {\bf Y} \right].$$
Then, $F_{\bf Y}(x)$ is nondecreasing in $x$ a.s. and $F^{-1}_{\bf Y}(t) = \inf\{ x \in \mathbb{R}: F_{\bf Y}(x) \geq t \}$ exists a.s. Moreover, if $U_1, U_2$ are two independent Uniform(0,1) random variables, independent of ${\bf Y}$, then
$$X_1 = F_{\bf Y}^{-1}(U_1) \qquad \text{and} \qquad X_2 = F_{\bf Y}^{-1}(U_2)$$
are identically distributed and conditionally independent given ${\bf Y}$, and 
$$(X_i, Y_1, Y_2, \dots) \stackrel{\mathcal{D}}{=} (X, Y_1, Y_2, \dots), \qquad i = 1,2.$$ 
\end{lemma}

\begin{proof}
That $F_{\bf Y}(x)$ is nondecreasing in $x$ a.s. follows from the fact that $1(X \leq x)$ is nondecreasing. Moreover, the pseudo inverse $F^{-1}_{\bf Y}(t)$ is well defined for all $t \in \mathbb{R}$ and satisfies $F_{\bf Y}( F^{-1}_{\bf Y}(t)) = t$ for all $t$. Now consider the two random variables $X_1$ and $X_2$ from the statement of the lemma.  Then, for any $x_1, x_2 \in \mathbb{R}$,
\begin{align*}
&E\left[ \left. 1(X_1 \leq x_1, \, X_2 \leq x_2) \right| {\bf Y} \right] \\
&= E\left[ \left. 1(F_{\bf Y}^{-1}(U_1) \leq x_1) 1(F_{\bf Y}^{-1}(U_2) \leq x_2) \right| {\bf Y} \right] \\
&= E\left[ \left. 1(F_{\bf Y}^{-1}(U_1) \leq x_1) \right| {\bf Y} \right] E\left[ \left. 1(F_{\bf Y}^{-1}(U_2) \leq x_2) \right| {\bf Y} \right] \\
&\hspace{40mm}  \text{(since $U_1$ and $U_2$ are independent)} \\
&= E\left[ \left. 1(X_1 \leq x_1) \right| {\bf Y} \right] E\left[ \left. 1(X_2 \leq x_2) \right| {\bf Y} \right],
\end{align*}
which shows the conditional independence given ${\bf Y}$. Furthermore, for any $x \in \mathbb{R}$,
\begin{align*}
E\left[ \left. 1(X_i \leq x) \right| {\bf Y} \right] &= E\left[ \left. 1(F_{\bf Y}^{-1}(U_i) \leq x) \right| {\bf Y} \right] \\
&= E\left[ \left. 1(U_i \leq F_{\bf Y}(x)) \right| {\bf Y} \right] = F_{\bf Y}(x) = E\left[ \left. 1(X \leq x) \right| {\bf Y} \right]
\end{align*}
for $i = 1,2$. Hence, $X_1$ and $X_2$ have the same conditional distribution as $X|{\bf Y}$, from where it follows that for any $A \subseteq \mathbb{R}^\infty$ and $i = 1,2$,
\begin{align*}
P\left( (X_i, Y_1, Y_2, \dots) \in A \right) &= E\left[ E\left[ \left. 1( (X_i, Y_1, Y_2, \dots) \in A) \right| {\bf Y} \right] \right] \\
&= E\left[ E\left[ \left. 1( (X, Y_1, Y_2, \dots) \in A) \right| {\bf Y} \right] \right] \\
&= P\left( (X, Y_1, Y_2, \dots) \in A \right).
\end{align*}
\end{proof}

We now derive as a corollary the strict positivity of the constant in Theorem 4.6 of \cite{Jel_Olv_12b} for the general case. Note that the result holds under no additional assumptions beyond those required in that theorem, in other words, we only require $E\left[ \left( \sum_{i=1}^N |C_i| \right)^{\alpha} \right] < \infty$ when $\alpha > 1$ and $E\left[ \left( \sum_{i=1}^N |C_i|^{\alpha/(1+\epsilon)} \right)^{1+\epsilon}  \right] < \infty$ for some $\epsilon > 0$ (which does not imply $\rho_{\alpha+\epsilon} < \infty$) when $0 < \alpha \leq 1$; compare these to conditions (C) and (A) in \cite{Alsm_Dam_Ment_13}, respectively. We also point out that Theorem 4.6 of \cite{Jel_Olv_12b} does not require the existence of the first root $0 < \upsilon < \alpha$ of the equation $E\left[ \sum_{i=1}^N |C_i|^\upsilon \right] = 1$, but only the derivative condition $0 < E\left[ \sum_{i=1}^N |C_i|^\alpha \log |C_i| \right] < \infty$.   Finally, our main result on the asymptotic behavior of the minimal/endogenous solution $R$ to the maximum equation (Theorem~\ref{T.MaximumRecursion}) gives that $H_L > 0$ provided $Q$ is not a deterministic function of the weights $\{C_i\}$ and assuming all the other conditions in Theorem 4.6 of \cite{Jel_Olv_12b} are satisfied.

\begin{cor} \label{C.StrictPositivity}
Under the assumptions of Theorem 4.6 in \cite{Jel_Olv_12b} and provided $Q$ is not a deterministic function of $(N, C_1, C_2, \dots)$, we have
$$P( |R| > t ) \sim K t^{-\alpha}, \qquad t \to \infty,$$
where $0 < K < \infty$ and 
$$K = \frac{E\left[ \left| \sum_{i=1}^N C_i R_i + Q \right|^\alpha - \sum_{i=1}^N |C_i R_i|^\alpha \right]}{\alpha E\left[ \sum_{i=1}^N |C_i|^\alpha \log|C_i| \right]}.$$
\end{cor}

\begin{remark}
Note that the same arguments used in the proof of Proposition~\ref{P.Symmetry} work, for any choice of $Q \not\equiv 0$, if the weights $\{C_i\}$ are symmetric, in which case the strict positivity of $H_L$ holds without any assumptions on $Q$. 
\end{remark}

\begin{proof}
Let $\{ (Q_{\bf i}, N_{\bf i}, C_{({\bf i}, 1)}, C_{({\bf i}, 2)}, \dots) \}_{{\bf i} \in U}$ be a  sequence of i.i.d. vectors and construct its corresponding random variable 
$$R = \sum_{k=0}^\infty \sum_{{\bf i} \in A_k} \Pi_{\bf i} Q_{\bf i}.$$
Now use this sequence and Lemma \ref{L.Construction} to construct a second i.i.d. sequence $\{ (\hat Q_{\bf i}, N_{\bf i}, C_{({\bf i}, 1)}, C_{({\bf i}, 2)}, \dots) \}_{{\bf i} \in U}$ where
$$(\hat Q, N, C_1, C_2, \dots)  \stackrel{\mathcal{D}}{=} (- Q, N, C_1, C_2, \dots)$$
and such that $\hat Q_{\bf i}$ and $Q_{\bf i}$ are conditionally independent given $(N_{\bf i}, C_{({\bf i},1)}, C_{({\bf i}, 2)}, \dots)$ for all ${\bf i} \in U$. Denote by $\hat R$ the corresponding process
$$\hat R = \sum_{k=0}^\infty \sum_{{\bf i} \in A_k} \Pi_{\bf i} \hat Q_{\bf i},$$
and note that $|R| \stackrel{\mathcal{D}}{=} |\hat R|$. Next define 
$$\overline{R} = \frac{R + \hat R}{2} =  \sum_{k=0}^\infty \sum_{{\bf i} \in A_k} \Pi_{\bf i} \left( \frac{Q_{\bf i} + \hat Q_{\bf i}}{2} \right) \triangleq \sum_{k=0}^\infty \sum_{{\bf i} \in A_k} \Pi_{\bf i} \overline{Q}_{\bf i},$$
and observe that $\overline{R}$ satisfies the conditions of Proposition \ref{P.Symmetry}, and therefore
$$P(|\overline{R}| > t ) \geq \frac{1}{2} P\left( \max_{{\bf i} \in \mathcal{T}} |\Pi_{\bf i} \overline{Q}_{\bf i} | > t \right).$$
Note that the assumption that $Q$ is not a deterministic function of $(N, C_1, C_2, \dots)$ implies that $\overline{Q} \not\equiv 0$. Moreover, by Theorem \ref{T.MaximumRecursion},
$$P\left( \max_{{\bf i} \in \mathcal{T}} |\Pi_{\bf i} \overline{Q}_{\bf i} | > t \right) \sim \overline{H} t^{-\alpha}$$
as $t \to \infty$ for some constant $0 < \overline{H} < \infty$. The last step is to note that
$$P( |\overline{R}| > t ) \leq P\left( |R| + |\hat R| > 2 t \right) \leq P( |R| > t ) + P(|\hat R| > t) = 2 P(|R| > t).$$

\end{proof}

\bibliographystyle{plain}

\end{document}